\documentclass[11pt,a4paper]{article}

\usepackage{amsmath,amssymb,amsthm,graphicx,url,color}

\evensidemargin=\oddsidemargin

\usepackage{enumitem}
\setlist{itemsep=-0.5pt,topsep=.6ex,leftmargin=5ex} 

\theoremstyle{definition}
\newtheorem{theorem}{Theorem}

\newtheorem{proposition}[theorem]{Proposition}
\newtheorem{definition}{Definition}
\newtheorem{example}{Example}

\def\<{\langle}
\def\>{\rangle}
\def\...{,\ldots,}
\def\refeq#1{{(\ref{eq:#1})}}
\def\sgn{\mathop{\rm sgn}}

\def\conv{\mathop{\rm conv}}
\def\cone{\mathop{\rm cone}}
\def\rint{\mathop{\rm rint}}
\def\cons{\mathop{\rm cons}}
\def\argmax{\mathop{\rm argmax}}
\def\argmin{\mathop{\rm argmin}}
\def\d{{\rm d}}

\def\bb#1{\mathbb{#1}} 

\title{On Coordinate Minimization of Convex Piecewise-Affine Functions}
\author{Tom\'a\v s Werner\\[1ex]
Dept.\ of Cybernetics, Faculty of Electrical Engineering, Czech Technical Univ.\ in Prague\\[1ex]
Research Report CTU--CMP--2017--05}
\date{September 14, 2017}

\begin{document}

\maketitle

\begin{abstract}
A popular class of algorithms to optimize the dual LP relaxation of
the discrete energy minimization problem (a.k.a.\ MAP inference in
graphical models or valued constraint satisfaction) are convergent
message-passing algorithms, such as max-sum diffusion, TRW-S, MPLP and
SRMP. These algorithms are successful in practice, despite the fact
that they are a version of coordinate minimization applied to a convex
piecewise-affine function, which is not guaranteed to converge to a
global minimizer. These algorithms converge only to a local minimizer,
characterized by local consistency known from constraint
programming. We generalize max-sum diffusion to a version of
coordinate minimization applicable to an arbitrary convex
piecewise-affine function, which converges to a local consistency
condition. This condition can be seen as the sign relaxation of the
global optimality condition.
\end{abstract}

\section{Introduction}

{\em Coordinate minimization\/} is an iterative method for
unconstrained optimization, which in every iteration minimizes the
objective function over a single variable while keeping the remaining
variables fixed. For some functions, this simple method converges to a
global minimum. This class includes differentiable convex functions
\cite[\S2.7]{Bertsekas99} and convex functions whose
non-differentiable part is separable~\cite{Tseng:2001}. For general
non-differentiable convex functions, the method need not converge to a
global minimum. 

For large-scale non-differentiable convex optimization, coordinate
minimization may be an acceptable option despite its inability to find
a global minimum. An example is dual LP relaxations of some NP-hard
combinatorial optimization problems, such as {\em discrete energy
  minimization\/} \cite{Szeliski06,Kappes-study-2015} (also known as
{\em MAP inference in graphical models\/}~\cite{Wainwright08} or {\em
  valued constraint satisfaction\/}
\cite{Meseguer06,Zivny-VCSPbook-2012}) and also some other
problems~\cite{Swoboda-CVPR-2017}. This dual LP relaxation leads to
the unconstrained minimization of a convex piecewise-affine function.
A number of algorithms for solving this dual LP relaxation have been
proposed. One class of algorithms, sometimes referred to as {\em
  convergent message passing\/}
\cite{Wainwright08,Globerson08,Hazan-UAI08,Meltzer-UAI-2009}, consists
of various forms of (block-)coordinate minimization applied to various
forms of the dual. Examples are max-sum diffusion
\cite{Kovalevsky-diffusion,Schlesinger-2011,Werner-PAMI07,Werner-PAMI-2010},
TRW-S~\cite{Kolmogorov06}, MPLP~\cite{Globerson08,Sontag-UAI-2012},
and SRMP~\cite{Kolmogorov-PAMI-2015}. Besides these, several
large-scale convex optimization methods converging to a global minimum
have been applied to the problem, such as subgradient methods
\cite{Schlesinger07-subgrad,Komodakis-PAMI-2011}, bundle
methods~\cite{Savchynskyy-CVPR-2012}, alternating direction method of
multipliers~\cite{Martins:ICL:2011}, and adaptive diminishing
smoothing~\cite{SavchynskyySKS12}. For practical problems from
computer vision, it was observed~\cite{Kappes-study-2015} that
convergent message-passing methods converge faster than these global
methods and their fixed points are often not far from global minima,
especially for sparse instances.

We ask whether convergent message-passing algorithms can be
reformulated to become applicable to a wider class of
non-differentiable convex functions than those arising from the above
dual LP relaxations. In other words, whether these algorithms can be
studied independently of any combinatorial optimization problems.  In
this paper, we make a step in this direction and generalize max-sum
diffusion to an algorithm applicable to an arbitrary convex
piecewise-affine function.



Consider an objective function in the form of a pointwise maximum of
affine functions. In this case, univariate minimizers in each
iteration of coordinate minimization need not be unique and therefore
some rule must be designed to choose a unique minimizer. We show that
for a certain natural deterministic rule, fixed points of coordinate
minimization can be poor. Imitating max-sum diffusion, we propose a
better rule which we call the {\em unique rule\/}: in each iteration,
minimize the maximum of only those affine functions that depend on the
current variable. With this rule, univariate minimizers are unique and
fixed points of the algorithm satisfy a natural condition, the {\em
  sign relaxation\/} of the true optimality condition. This can be
seen as a {\em local consistency\/} as used in constraint
programming~\cite{Bessiere06}, which is known to characterize fixed
points of many algorithms to solve the dual LP relaxation, namely
message-passing algorithms
\cite{Werner-PAMI07,Werner-PAMI-2010,Kolmogorov06,Kolmogorov-PAMI-2015}
and the algorithm \cite{Koval76,Cooper-AI-2010}.

Little is known theoretically about convergence properties of
message-passing algorithms. Although max-sum diffusion and TRW-S were
always observed to converge to a fixed point, this was never
proved. For finite-valued energy terms, it has been proved that every
accumulation point of TRW-S and SRMP satisfy local consistency
\cite{Kolmogorov06,Kolmogorov-PAMI-2015}. Under the same assumption, a
stronger result has been proved for max-sum
diffusion~\cite{Schlesinger-2011}: a quantity that measures how much
the local consistency condition is violated converges to zero (note,
this is still a weaker result that convergence to a fixed
point). Nothing is known theoretically about convergence rates of
max-sum diffusion, TRW-S, MPLP and SRMP.

We revisit the proof from~\cite{Schlesinger-2011} to show that, under
a certain assumption, during coordinate minimization with the unique
rule the above quantity converges to zero. In contrast to max-sum
diffusion, we show there are objective functions for which the
algorithm has no fixed point or even no point satisfying sign consistency.

\section{Minimizing a Pointwise Maximum of Affine Functions}

Consider a function $f{:}\ \bb R^n\to\bb R$ that is a pointwise
maximum of affine functions, i.e.,
\begin{equation}
f(x) = \max_{i\in[m]}(a_i^Tx+b_i)
\label{eq:fun}
\end{equation}
where $a_1\...a_m\in\bb R^n$ and $b_1\...b_m\in\bb R$, and we denote
$[m]=\{1\...m\}$. For brevity, we will further on abuse symbols `max'
and `argmax' and denote, for $y\in\bb R^m$,
\begin{subequations}
\begin{align}
\max y &= \max_{i\in[m]}y_i, \label{eq:max} \\
\argmax y &= \argmax_{i\in[m]} y_i 
= \{\, i\in[m] \mid y_i = \max y \,\} . \label{eq:argmax}
\end{align}
\end{subequations}
Now function~\refeq{fun}~can be written simply as
\begin{equation}
f(x) = \max(Ax+b) ,
\label{eq:fun'}
\end{equation}
where matrix $A=(a_{ij})\in\bb R^{m\times n}$ has rows
$a_1^T\...a_m^T$ and $b=(b_1\...b_m)\in\bb R^m$.


The well-known condition for a global minimum of a convex function
$f{:}\ \bb R^n\to\bb R$ is that $0\in\partial f(x)$. For
function~\refeq{fun}, the subdifferential reads
\begin{equation}
\partial f(x) = \conv\{\, a_i \mid i\in \argmax(Ax+b) \,\} .
\label{eq:fun-subdiff}
\end{equation}
In more detail, the following holds:

\begin{proposition}
For function~$f$ given by~\refeq{fun} it holds:
\begin{itemize}
\item $f$ is bounded from below iff $0\in\conv\{a_1\...a_m\}$.
\item $x$ is a minimizer of~$f$ iff $0\in\conv\{\,a_i\mid i\in\argmax(Ax+b)\,\}$.
\item If $f$~is bounded from below, it has a minimizer~$x$ such that
$0\in\rint\conv\{\,a_i\mid i\in\argmax(Ax+b)\,\}$.
\end{itemize}
\end{proposition}

\begin{proof}
The minimization of function~\refeq{fun} can be written as the linear
program
\begin{equation}
\min\{\, z \mid Ax+b\le z, \; x\in\bb R^n, \; z\in\bb R \,\} .
\label{eq:primal}
\end{equation}
The dual to this linear program reads
\begin{equation}
\max\{\, b^T\lambda \mid \lambda\in\bb R^m, \; \lambda\ge0, \; 1^T\lambda=1, \; A^T\lambda=0\,\} .
\label{eq:dual}
\end{equation}
The primal~\refeq{primal} is always feasible. By strong duality,
$f$~is bounded from below iff the dual~\refeq{dual}
is feasible. By complementary slackness, $x$ and~$\lambda$ are both
optimal iff $\lambda_i=0$ for every $i\notin\argmax(Ax+b)$.  If $f$~is
bounded from below, by strict complementary slackness there exist
optimal $x$ and~$\lambda$ such that $\lambda_i>0$ iff
$i\in\argmax(Ax+b)$. Note that
$A^T\lambda=\sum_{i\in[m]}\lambda_ia_i$. Note that, for any $k\in\bb N$ and $a_1\...a_k\in\bb R^n$,
\begin{subequations}
\begin{align}
0\in\conv\{a_1\...a_k\} \quad&\Longleftrightarrow\quad \exists \lambda\in\bb R^k{:}\ \textstyle \lambda_i\ge0, \; \sum_i\lambda_i=1, \; \sum_i\lambda_ia_i=0 , \label{eq:conv-lambda} \\
0\in\rint\conv\{a_1\...a_k\} \quad&\Longleftrightarrow\quad \exists \lambda\in\bb R^k{:}\ \textstyle \lambda_i>0, \; \sum_i\lambda_i=1, \; \sum_i\lambda_ia_i=0 . \label{eq:rint-conv-lambda}
\end{align}
\end{subequations}
All three claims are now obvious.
\end{proof}





%
\begin{proposition}
\label{th:bounded-y}
A non-empty set $\{\, Ax \mid x\in\bb R^n, \; Ax\le b\,\}$ is bounded
iff $0\in\rint\conv\{a_1\...a_m\}$.
\end{proposition}

\begin{proof}
As the set is non-empty, it is bounded iff the linear program
\[
\min\{\, c^TAx \mid x\in\bb R^n, \; Ax\le b\,\}
\]
is bounded for every $c\in\bb R^m$. This holds iff the dual linear
program
\[
\max\{\,-b^T\lambda\mid A^T(\lambda+c)=0,\;\lambda\ge0\,\}
\]
is feasible for every $c\in\bb R^m$. 
We show that
\begin{equation}
\forall c\in\bb R^m \; \exists \lambda\ge0{:}\ A^T(\lambda+c)=0
\quad\Longleftrightarrow\quad
\exists \lambda'>0{:}\ A^T\lambda'=0 .
\label{eq:bounded-y-proof}
\end{equation}
\begin{itemize}
\item To prove $\Rightarrow$, take $c=1$ and let $\lambda\ge0$ be such that
$A^T(\lambda+c)=0$. Then set $\lambda'=\lambda+c=\lambda+1$.
\item To prove $\Leftarrow$, set $\lambda=\alpha\lambda'-c$ where
$\alpha>0$ is arbitrary such that $\lambda\ge0$. Such~$\alpha$ clearly exists.
\end{itemize}
Note that $\lambda'$~can be multiplied by a positive scale to satisfy
$1^T\lambda'=1$.  By~\refeq{rint-conv-lambda}, the right-hand
statement in~\refeq{bounded-y-proof} is thus equivalent to
$0\in\rint\conv\{a_1\...a_m\}$.
\end{proof}

\section{Sign Relaxation of the Optimality Condition}

We said that deciding if $x\in\bb R^n$ is a minimizer of
function~\refeq{fun} requires deciding if the convex hull of a finite
set of vectors contains the origin. Deciding this condition for large
problems may be infeasible. We define the {\em sign relaxation\/} of
this condition which is cheaper to decide, obtained by considering
only the signs of the vectors' components and dropping their
magnitudes.

To describe the key idea, suppose that the convex hull of some vectors
$a_1\...a_m\in\bb R^n$ contains the origin,
$0\in\conv\{a_1\...a_m\}$. That is, there are numbers
$\lambda_1\...\lambda_m\ge0$ such that
\begin{subequations}
\begin{align}
\sum_{i\in[m]}\lambda_i&=1 , \label{eq:0inconv:1} \\
\sum_{i\in[m]}\lambda_ia_{ij}&=0, \quad j\in[n] . \label{eq:0inconv:2}
\end{align}
\end{subequations}
Let us relax these conditions, considering only the signs
$\sigma_i=\sgn\lambda_i\in\{0,1\}$ and
$s_{ij}=\sgn a_{ij}\in\{-1,0,1\}$. Equality~\refeq{0inconv:1} implies
that at least one of the numbers $\sigma_1\...\sigma_m$ equals~1. For
each $j\in[n]$, equality~\refeq{0inconv:2} implies that the numbers
$\sigma_1s_{1j}\...\sigma_ms_{mj}$ either are all zero, or some are positive and some
negative. One way to write this is as follows:
\begin{subequations}
\label{eq:signCSP}
\begin{align}
\exists i\in[m]{:}\ \sigma_i=1 , & \label{eq:signCSP:1} \\
(\exists i\in[m]{:}\ \sigma_is_{ij}=-1) \Leftrightarrow (\exists i'\in[m]{:}\ \sigma_{i'}s_{i'j}=1) , & \quad j\in[n] . \label{eq:signCSP:2}
\end{align}
\end{subequations}
Thus, it is necessary for $0\in\conv\{a_1\...a_m\}$ that
there exist some $\sigma_1\...\sigma_m\in\{0,1\}$ satisfying~\refeq{signCSP}.

This can be seen as a {\em constraint satisfaction problem\/} (CSP)
\cite{Mackworth91,Freuder06} with $m$~binary variables and $n+1$
constraints. This particular CSP class can be solved by enforcing {\em
  (generalized) arc consistency\/}~\cite{Bessiere06}. Suppose that for
some~$j$, the signs $s_{1j}\...s_{mj}$ are, say, all non-negative and
some of them is positive. Then constraint~\refeq{signCSP:2} enforces
that for all~$i$ for which $s_{ij}\neq0$ we have
$\sigma_i=0$. Repeating this for various coordinates~$j$ progressively
sets some~$\sigma_i$ to zero. If finally $\sigma_i=0$ for all
$i\in[n]$, constraint~\refeq{signCSP} is violated and the CSP has no
solution. Otherwise, the CSP has a solution.

Since we believe that the described concept of sign relaxation might
have a wider applicability in non-differential convex optimization, we
further develop it in more detail and greater generality than is needed
in this paper. Namely, we consider the sign relaxation of the
condition $0\in\conv X$ for any (not necessarily finite) set
$X\subseteq\bb R^n$. This is straightforward because the sign
relaxation of this condition depends only on the set $\{\,\sgn a\mid
a\in X\,\}\subseteq\{-1,0,1\}^n$ which is finite.  Here, for a vector
$a=(a_1\...a_n)\in\bb R^n$, we denoted $\sgn a=(\sgn a_1\...\sgn
a_n)\in\{-1,0,1\}^n$.

\begin{theorem}
\label{th:rint}
Let $X\subseteq\bb R^n$ and $a\in\conv X$. Then
$a\in\rint\conv(X\cap F)$ where $F$~is the intersection of all faces
of $\conv X$ that contain\footnote{More precisely, it can be shown
  that $Y=X\cap F$ is the greatest (with respect to partial ordering by
  inclusion) subset of~$X$ such that $a\in\rint\conv Y$.}~$a$.
\end{theorem}

\begin{proof}
It can be shown that for any face~$F$ of $\conv X$ it holds that\footnote{We
  omit the proof of this claim. The proof is obvious if $X$~is finite
  and hence $\conv X$ is a convex polytope. For infinite~$X$, recall
  \cite{Hiriart-book-2004} that a {\em face\/} of a convex set~$C$ is
  a convex set $F\subseteq C$ such that every line segment from~$C$
  whose relative interior has a non-empty intersection with~$F$ is
  contained in~$F$.}  $\conv(X\cap F)=F$.  So it suffices to show that
$a\in\rint F$. For contradiction, suppose $a$~is a relative boundary
point of~$F$. But every relative boundary point of a face is contained
in some subface of the face, so $F$~cannot be the intersection of all
faces of $\conv X$ containing~$a$.
\end{proof}





%
%
%

\begin{definition}
\label{def:SC}
A set $S\subseteq\{-1,0,1\}^n$ is {\em consistent in
  coordinate\/} $j\in[n]$ if it holds that
\begin{equation}
\exists s\in S{:}\ s_j=-1 \quad\Longleftrightarrow\quad \exists t\in S{:}\ t_j=1 .
\label{eq:SC}
\end{equation}
Set~$S$ is {\em consistent\/} if it is consistent in
every coordinate $j\in[n]$.
\end{definition}

In particular, note that the sets $\emptyset$ and $\{0\}$ (where 0~denotes
the vector with $n$~zeros) are consistent.

\begin{theorem}
\label{th:rint-scons}
Let $X\subseteq\bb R^n$. If $0\in\rint\conv X$, then the set
$\{\,\sgn a\mid a\in X\,\}$ is consistent.
\end{theorem}

\begin{proof}
Let $0\in\rint\conv X$. Then the projection of the set $\rint\conv X$
onto each coordinate axis is either the set $\{0\}$ or an interval
containing zero as its interior point. Noting that projections commute
with the convex hull operator, this means for every $j\in[n]$ we have that
\begin{equation}
\exists a\in X{:}\ a_j<0 \quad\Longleftrightarrow\quad \exists b\in X{:}\ b_j>0 .
\label{eq:SC'}
\end{equation}
This is equivalent to condition~\refeq{SC} for the set $\{\,\sgn a\mid a\in X\,\}$.
\end{proof}

\begin{theorem}
\label{th:coord-consist>consist}
Let $X\subseteq\bb R^n$. If $0\in\conv X$, then the set
$\{\,\sgn a\mid a\in X\,\}$ has a non-empty consistent subset.
\end{theorem}

\begin{proof}
If $0\in\conv X$, then by Theorem~\ref{th:rint} there is
$Y\subseteq X$ such that $0\in\rint\conv Y$. Thus, by
Theorem~\ref{th:rint-scons}, the set $\{\,\sgn a\mid a\in Y\,\}$ is
consistent. This set is non-empty and it is a subset of
$\{\,\sgn a\mid a\in X\,\}$.
\end{proof}

Theorem~\ref{th:coord-consist>consist} gives a necessary condition for
a set $X\subseteq\bb R^n$ to satisfy $0\in\conv X$. This is what we
called the sign relaxation of $0\in\conv X$. This condition is of
course not sufficient, e.g., for $X=\{(1,-2),(-2,1)\}$ the set
$\{\,\sgn a\mid a\in X\,\}=\{(1,-1),(-1,1)\}$ is consistent but
$0\notin\conv X$. 

Next we develop the concept of {\em consistency closure\/}, analogous
to arc consistency closure~\cite{Bessiere06}.



\begin{theorem}
\label{th:SC-union}
Let the sets $S,T\subseteq\{-1,0,1\}^n$ be consistent. Then the
set $S\cup T$ is consistent.
\end{theorem}

\begin{proof}
Let $s\in S\cup T$, so $s\in S$ or $s\in T$. Suppose $s\in S$.
Suppose for some $j\in[n]$ we have, say, $s_j<0$. As $S$~is
consistent, there is $t\in S$ such that $s_j>0$. But
$t\in S\cup T$.
\end{proof}

\begin{definition}
\label{def:SCclosure}
The {\em consistency closure\/} of a set $S\subseteq\{-1,0,1\}^n$
is the greatest (with respect to partial ordering by inclusion)
consistent subset of~$S$, i.e., it is the union of all
consistent subsets of~$S$. We will denote it by $\cons S$.
\end{definition}

\begin{theorem}
The $\cons$ operator satisfies the axioms of a closure operator, i.e., it is
\begin{itemize}
\item intensive ($\cons S \subseteq S$),
\item idempotent ($\cons \cons S = \cons S$),
\item non-increasing ($S\subseteq T\;\Longrightarrow\;\cons S \subseteq \cons T$).
\end{itemize}
\end{theorem}

\begin{proof}
Intensivity and idempotency are immediate from Definition~\ref{def:SCclosure}.
To show the non-increasing property, let $S\subseteq T$. By
intensivity, we have $\cons S\subseteq T$. But $\cons S$ is
consistent, thus by Definition~\ref{def:SCclosure} it must be a subset
of $\cons T$.
\end{proof}

%

\begin{definition}
\label{def:SCalg}
{\em Enforcing consistency\/} of a 
set $S\subseteq\{-1,0,1\}^n$ is the algorithm that repeats the following
iteration:
\begin{enumerate}
\item choose an arbitrary coordinate $j\in[n]$ in which $S$~is not
consistent,
\item remove from~$S$ all elements~$s$ such that $s_j\neq0$.
\end{enumerate}
If $S$ is consistent, the algorithm stops.
\end{definition}

\begin{theorem}
Enforcing consistency of any set $S\subseteq\{-1,0,1\}^n$
yields $\cons S$.
\end{theorem}

\begin{proof}
The algorithm creates a sequence of sets
$S=S_0\supset S_1\supset\cdots\supset S_K$, where $S_K$~is consistent
and $K$~is the number of iterations. This sequence is given
recurrently by
$S_k=S_{k-1}\setminus\{\,s\in S_{k-1}\mid s_{j_k}\neq0\,\}$ where each
$j_k\in[n]$ is such that $S_{k-1}$ is not consistent in
coordinate~$j_k$.

Let a set $T\subseteq S$ be consistent. We show by induction that
$T\subseteq S_k$ for all~$k$. Assume $T\subseteq S_{k-1}$. Since
$S_{k-1}$~is not consistent in coordinate~$j_k$, 
it follows from Definition~\ref{def:SC} that $s_{j_k}=0$ for every
$s\in T$. Therefore
$T\subseteq S_{k-1}\setminus\{\,s\in S_{k-1}\mid s_{j_k}\neq0\,\} =
S_k$.

We have shown that every consistent subset of~$S$ is contained
in~$S_K$, which is itself consistent. This means that
$S_K=\cons S$.
\end{proof}

To conclude, the sign relaxation of the condition $0\in\conv X$ is
that $\cons\{\,\sgn a\mid a\in X\,\}\neq\emptyset$.

\section{Coordinate Minimization}


{\em Coordinate minimization\/} of a function $f{:}\ \bb R^n\to\bb R$
is a method that, starting from an initial point $x=(x_1\...x_n)$, repeats the
following iteration:
\begin{enumerate}
\item choose $j\in[n]$,
\item choose $\displaystyle x^*_j\in\argmin_{x_j\in\bb R}f(x)$,
\item set $x_j\gets x^*_j$.
\end{enumerate}
Since the choices in the first two steps are not specified, this does
not define a single algorithm but rather a class of algorithms (a
`method'). 

Further on, we focus on applying coordinate minimization to functions
of the form~\refeq{fun}.

For functions of the form~\refeq{fun}, in Step~2 of coordinate
minimization the univariate minimizer $x^*_j$ is in general not
unique. Therefore, some rule must be adopted to choose a unique
minimizer. Let us restrict ourselves to deterministic rules. We show
that some rules can behave poorly. Consider the following rule: choose
the element of the set $\argmin_{x_j\in\bb R}f(x)$ that is nearest
to~$x_j^{\rm prev}$, the $j$-th variable from the previous
iteration. We refer to this as the {\em proximal rule\/}\footnote{This
  can be indeed seen as a proximal regularization of coordinate
  minimization: instead of function~$f$ we minimize the function
  $g(x,y)=f(x)+\mu\|x-y\|^2$ for a small $\mu>0$. For this function,
  the univariate minimization would read
  $\argmin_{x_j\in\bb R}(f(x_1\...x_n)+\mu(x_j-x_j^{\rm prev})^2)$,
  which has the unique minimizer given by the proximal rule. Cf.\
  \cite[Exercise~2.7.2]{Bertsekas99}.}.

\begin{example}
\label{ex:prox}
Consider the function
\begin{equation}
f(x_1,x_2,x_3) = \max\{\,x_2-x_3,\,x_3-x_1,\,x_1-x_2\,\} ,
\label{eq:fun-prox}
\end{equation}
which has minimum value~0, attained for any $x_1=x_2=x_3$. The point
$x=(x_1,x_2,x_3)=(2,1,0)$ is fixed for coordinate minimization with
the proximal rule, with value $f(x)=\max\{1,-2,1\}=1$. The
subdifferential at this point is $\partial f(x) = \conv X$ where
$X=\{(0,1,-1),(1,-1,0)\}$. Using the algorithm from
Definition~\ref{def:SCalg} we find that
$\cons\{\,\sgn a \mid a\in
X\,\}=\cons\{(0,1,-1),(1,-1,0)\}=\emptyset$.
Thus, point~$x$ does not satisfy the sign relaxation of the condition
$0\in\partial f(x)$.
\end{example}

Imitating max-sum diffusion, we propose a better rule, which we call
the {\em unique rule\/}: in Step~2 of coordinate minimization, rather
than minimizing the maximum of all affine functions, minimize the
maximum of only those affine functions that depend on
variable~$x_j$. That is, minimize the function
\begin{equation}
\phi_j(x) = \max_{i\mid a_{ij}\neq0} (a_i^Tx+b_i)
\label{eq:phi}
\end{equation}
where $\max_{i\mid a_{ij}\neq0}$ denotes maximization over all
$i\in[m]$ such that $a_{ij}\neq0$.
Further on, we will assume that the set $\{\,\sgn a_i\mid i\in[m]\,\}$ is
consistent.
Under this assumption, for every $j\in[n]$ and every
$x_1\...x_{j-1},x_{j+1}\...x_n\in\bb R$, the univariate function
$x_j\mapsto\phi_j(x)$ has exactly one minimizer, $x^*_j$. This
minimizer is the unique solution of the equation
\begin{equation}
\max_{i\mid a_{ij}<0} (a_i^Tx+b_i) = \max_{i\mid a_{ij}>0} (a_i^Tx+b_i) .
\label{eq:fixedpt-x}
\end{equation}
To summarize, the iteration of the algorithm for coordinate $j\in[n]$
adjusts variable~$x_j$ to satisfy~\refeq{fixedpt-x}, keeping the
other variables unchanged.  A {\em fixed point\/} of the
algorithm\footnote{Note, such a point is a Nash equilibrium for
  penalty functions $\phi_1\...\phi_n{:}\ \bb R^n\to\bb R$.} is a
point $x\in\bb R^n$ that satisfies~\refeq{fixedpt-x} for all
$j\in[n]$.

For functions in the form~\refeq{fun}, coordinate minimization can be
equivalently formulated in terms of the values $y=Ax+b$ of the affine
functions, i.e., instead of updating the numbers $x_1\...x_n$ we
update the numbers $y_1\...y_m$, while $x_1\...x_n$ are no longer
explicitly kept. We first set $y=Ax+b$, where $x$~is the initial
point, and then repeat the following iteration:
\begin{enumerate}
\item choose $j\in[n]$,
\item choose $\displaystyle x^*_j\in\argmin_{x_j\in\bb R}\max_{i\in[m]}(a_{ij}x_j+y_i)$,
\item set $y_i\gets a_{ij}x^*_j+y_i$ for all $i\in[m]$.
\end{enumerate}
Using the unique rule, in Step~2 we need to find the unique
minimizer~$x_j^*$ of the univariate function
$\max\limits_{i\mid a_{ij}\neq0}(a_{ij}x_j+y_i)$, which is the
solution of the equation
\begin{equation}
\max_{i\mid a_{ij}<0}(a_{ij}x_j+y_i) = \max_{i\mid a_{ij}>0}(a_{ij}x_j+y_i) .
\label{eq:fixedpt-y'}
\end{equation}
After the iteration, we thus have
\begin{equation}
\max_{i\mid a_{ij}<0}y_i = \max_{i\mid a_{ij}>0}y_i .
\label{eq:fixedpt-y}
\end{equation}
A fixed point of the algorithm is a point
$y\in\{\,A(x+x')+b \mid x'\in\bb R^n\,\}$ that
satisfies~\refeq{fixedpt-y} for all $j\in[n]$. We introduce the
following notations:
\begin{itemize}
\item Mapping $p_j{:}\ \bb R^m\to\bb R^m$ denotes the action of Steps~2 and~3 of
coordinate minimization with the unique rule, formulated in terms
of~$y$. That is, for $y\in\bb R^m$ and $j\in[n]$, $p_j(y)$ is
computed as follows: find the solution~$x_j^*$ of
equation~\refeq{fixedpt-y'} and then set $y_i\gets a_{ij}x^*_j+y_i$
for all $i\in[m]$.
\item Further on, we assume that coordinates~$j$ in Step~1 of
coordinate minimization are visited in the cyclic order. Let
$p=p_n\circ p_{n-1}\circ\cdots\circ p_2\circ p_1$ denote the action of
the algorithm for one cycle.
\item For $k\in\bb N$, let $p^k=p\circ\cdots\circ p$
($k$-times) denote the action of the algorithm for $k$~cycles.
\end{itemize}
In this notation, $y\in\bb R^m$ is a fixed point of the algorithm iff
$p_j(y)=y$ for every $j\in[n]$, which holds iff $p(y)=y$.

Next we give several examples of the algorithm's behavior.

\begin{example}
Recall that coordinate minimization is not guaranteed to find a global
minimum of a function of the form~\refeq{fun} because it is 
not differentiable. An example is the function
\[
f(x_1,x_2)=\max\{\,x_1-2x_2,\,x_2-2x_1\,\} ,
\]
which is unbounded but any point $x_1=x_2$ is fixed for coordinate
minimization. At any such point we have
$0\notin\partial f(x_1,x_2)=\conv X$ where $X=\{(1,-2),(-2,1)\}$. The
set $\{\,\sgn a\mid a\in X\,\}=\{(1,-1),(-1,1)\}$ is consistent.
There is no difference between the proximal rule and unique rule,
because univariate minimizers in Step~2 are unique. Coordinate
minimization converges in one iteration.
\end{example}


The unique rule is not worse than the proximal rule, in the sense that
every fixed point of coordinate minimization with the unique rule is
also a fixed point of coordinate minimization with the proximal
rule. The following example shows that the unique rule is in fact
strictly better.

\begin{example}
Let us return to Example~\ref{ex:prox} and apply coordinate
minimization with the unique rule to function~\refeq{fun-prox}. For
$j=1$, equation~\refeq{fixedpt-y'} reads $-x_1+y_2=x_1+y_3$ with the
solution $x_1^*=(y_2-y_3)/2$, thus $y$~is updated as
$y_2\gets -x_1^*+y_2=(y_2+y_3)/2$ and
$y_3\gets x_1^*+y_3=(y_2+y_3)/2$. Similarly we derive the iterations
for $j=2,3$. In summary, we have
\begin{align*}
p_1(y_1,y_2,y_3) &= (\,y_1,\,(y_2+y_3)/2,\,(y_2+y_3)/2\,) , \\
p_2(y_1,y_2,y_3) &= (\,(y_3+y_1)/2,\,y_2,\,(y_3+y_1)/2\,) , \\
p_3(y_1,y_2,y_3) &= (\,(y_1+y_2)/2,\,(y_1+y_2)/2,\,y_3\,) .
\end{align*}
We see that in every iteration, the algorithm takes a pair of the
numbers $y_1,y_2,y_3$ and replaces both of them with their
average. For any initial $y_1,y_2,y_3$, the sequence
$(p^k(y_1,y_2,y_3))_{k\in\bb N}$ converges to the point with
$y_1=y_2=y_3=(y_1+y_2+y_3)/3$, i.e., to a minimizer of
function~\refeq{fun-prox}.
\end{example}

The next example shows that there are functions for which coordinate
minimization with the unique rule has no fixed point and even no point
satisfying the sign-relaxed optimality condition.

\begin{example}
%
Let
\begin{equation}
f(x_1,x_2,x_3) = \max\{\, x_1-x_2-x_3, \, x_1+4, \, x_1+x_2+x_3, \, -x_1+x_2+2 \,\} .
\label{eq:fun-nofixed}
\end{equation}
This function is not bounded from below. System~\refeq{fixedpt-x},
defining the fixed point condition, reads
\begin{align*}
\max\{\, x_1-x_2-x_3, \, x_1+4, \, x_1+x_2+x_3 \,\} &= -x_1+x_2+2 \\
\max\{\, x_1+x_2+x_3, \, -x_1+x_2+2 \,\} &= x_1-x_2-x_3 \\
x_1+x_2+x_3 &= x_1-x_2-x_3
\end{align*}
The third equation implies $x_2+x_3=0$, thus the system simplifies to
\begin{align*}
\max\{\, x_1, \, x_1+4, \, x_1 \,\} &= -x_1+x_2+2 \\
\max\{\, x_1, \, -x_1+x_2+2 \,\} &= x_1
\end{align*}
The first equation simplifies to $x_1+4 = -x_1+x_2+2$, hence
$x_2=2x_1+2$. Plugging this to the second equation gives a
contradiction. This shows that the algorithm has no fixed point.

We shall show that there is even no point $x\in\bb R^n$ such that
$\cons\{\,\sgn a_i\mid i\in\argmax(Ax+b)\,\}\neq\emptyset$. For
contradiction, suppose it is so. That is, there is $x\in\bb R^n$ and
$I\subseteq[m]$ such that the set $\{\,\sgn a_i\mid i\in I\,\}$ is
consistent and $\emptyset\neq I\subseteq\argmax(Ax+b)$.  It can be
checked that the only non-empty subset of $[m]$ for which the set
$\{\,\sgn a_i\mid i\in I\,\}$ is consistent is $I=[m]$. But there is
no~$x$ such that $\argmax(Ax+b)=[m]$, i.e., at which all four affine
functions are active.


How will the algorithm behave in this case? For the initial point
$(x_1,x_2,x_3)=(0,0,0)$, the first three iterations of the algorithm
are
\[
\begin{array}{|c|ccc|cccc|}
\hline
j & x_1 & x_2 & x_3 & y_1 & y_2 & y_3 & y_4 \\
\hline
  & 0 & 0 & 0 &   0 & 4 & 0 & 2 \\
1 & -1 & 0 & 0 &  -1 & 3 & -1 & 3 \\
2 & -1 & -2 & 0 &  1 & 3 & -3 & 1 \\
3 & -1 & -2 & 2 & -1 & 3 & -1 & 1 \\
\hline
\end{array}
\]
The resulting values of~$y$ are the initial
values minus one. Every later cycle will again decrease them by one, therefore
algorithm therefore will decrease~$y$ unboundedly.
\end{example}

\begin{example}
Consider the function
\begin{equation}
f(x_1,x_2,x_3) = \max\{\, x_1-x_2-x_3, \, x_1+4, \, x_1+x_2+x_3, \, -x_1+x_2+2, \, 0 \,\} ,
\label{eq:fun-nofixed-0}
\end{equation}
which differs from~\refeq{fun-nofixed} only by the extra zero
function.  As shown in the previous example, the algorithm has no
fixed point and the first four components of the vector
$y=(y_1,y_2,y_3,y_4,y_5)$ will diverge. Since after a few iterations
$y_1,y_2,y_3,y_4$ become all negative but $y_5$~remains zero, the set
$\cons\{\,\sgn a_i\mid i\in\argmax y\,\}$ becomes non-empty.
\end{example}

The following theorem shows that every fixed point of coordinate
minimization with the unique rule satisfies the sign relaxation of the
condition $0\in\partial f(x)$.

\begin{theorem}
If \refeq{fixedpt-y}~holds for every $j\in[n]$, then the set
$\{\,\sgn a_i\mid i\in \argmax y\,\}$ is consistent.
\end{theorem}

\begin{proof}
For every $j\in[n]$, \refeq{fixedpt-y} implies that
\begin{equation}
\exists i\in \argmax y{:}\ a_{ij}<0 \quad\Longleftrightarrow\quad \exists i'\in \argmax y{:}\ a_{i'j}>0 .
\label{eq:argmax-SC}
\end{equation}
Indeed, if the common value of both sides in~\refeq{fixedpt-y} is
equal to [less than] $\max y$, both sides of~\refeq{argmax-SC} are
true [false]. Condition~\refeq{argmax-SC} is equivalent to
condition~\refeq{SC} for the set
$\{\,\sgn a_i\mid i\in \argmax y\,\}$.
\end{proof}

For an initial point
$y\in\bb R^m$, consider the sequence of vectors
$(p^k(y))_{k\in\bb N}$. Although we believe that, under some
reasonably weak assumptions, this sequence converges to a fixed point,
we are not able to prove this. Following~\cite{Schlesinger-2011}, we
formulate and prove a weaker result.

For $\epsilon\ge0$, let
\begin{equation}
\argmax\nolimits^\epsilon y 
= \{\,i\in[m]\mid y_i + \epsilon \ge \max y \,\}
\label{eq:argmax-eps}
\end{equation}
denote the set of $\epsilon$-maximal components of a vector
$y\in\bb R^m$. We now define function $e{:}\ \bb R^m\to\bb R_+$ by
\begin{equation}
e(y) = \inf\{\, \epsilon\ge0 \mid \cons\{\,\sgn a_i\mid i\in \argmax\nolimits^\epsilon y\,\}\neq\emptyset \,\} .
\label{eq:e}
\end{equation}
This function measures how much point~$y$ violates the
condition\footnote{In~\cite{Schlesinger-2011}, `consistency' means
  $\cons\{\,\sgn a_i\mid i\in \argmax\nolimits y\,\}\neq\emptyset$,
  rather than consistency in the sense of our
  Definition~\ref{def:SC}.}
$\cons\{\,\sgn a_i\mid i\in \argmax\nolimits y\,\}\neq\emptyset$,
which is the sign relaxation of the condition
$0\in\conv\{\,a_i\mid i\in\argmax y\,\}=\partial f(x)$.

\begin{theorem}
\label{th:e-converges}
Let $0\in\rint\conv\{a_1\...a_m\}$. Let $y\in\bb R^m$. Then $\lim\limits_{k\to\infty} e(p^k(y)) = 0$.
\end{theorem}

\begin{proof}
In the appendix.
\end{proof}

\subsection{Sum of Maxima of Affine Functions}

We have applied coordinate minimization to pointwise maximum of affine
functions~\refeq{fun}. But what if we want to apply it to a convex
piecewise-affine function in the form of a sum of pointwise maxima of
affine function, i.e.,
\begin{equation}
f(x) = \sum_i \max_j (a_{ij}^Tx+b_{ij}) ,
\label{eq:sum-max-fun}
\end{equation}
where $a_{ij}\in\bb R^n$ and $b_{ij}\in\bb R$. It turns out that
minimizing function~\refeq{sum-max-fun} can be easily transformed to
minimizing a function~\refeq{fun}. One way to do that is as follows.

\begin{theorem}
\label{th:sum-max-aux}
For every $\alpha_1\...\alpha_m\in\bb R$,
\begin{equation}
{1\over m}\sum_{i\in[m]} \alpha_i =
\min_{u_1+\cdots+u_m=0} \; \max_{i\in[m]}(\alpha_i+u_i) .
\label{eq:sum-max-aux}
\end{equation}
\end{theorem}

\begin{proof}
Clearly, at optimum all expressions under the maximum will have the
same value. Let us denote this common value by $b=\alpha_i+u_i$, so
$\alpha_i=b-u_i$. Now $\sum_i\alpha_i=mb-\sum_iu_i=mb$.
\end{proof}

Using Theorem~\ref{th:sum-max-aux}, minimizing
function~\refeq{sum-max-fun} can be transformed to minimizing the
function
\[
g(x,u) = \max_{i,j}(a_{ij}^Tx+b_{ij}+u_i)
\]
subject to $\sum_iu_i=0$. Minimization over~$u$ can be done in closed
form, interlacing iterations of coordinate minimization over~$x$.



\appendix

\section{Proof of Theorem~\ref{th:e-converges}}

\subsection{Properties of the Algorithm}
\label{sec:proof-prelim-alg}


\begin{proposition}
\label{th:monotone}
For every $y\in\bb R^m$ and $j\in[n]$ we have $\max p_j(y)\le \max y$.
\end{proposition}

\begin{proof}
This just says that coordinate minimization never increase the
objective function.
\end{proof}

For $y\in\bb R^m$ and $z\in\bb R$, we denote
\begin{equation}
I(y,z)=\{\,i\in[m]\mid y_i\ge z\,\} .
\label{eq:I}
\end{equation}
In particular, note that $\argmax y=I(y,\max y)$ and
$\argmax\nolimits^\epsilon y = I(y,\max y-\epsilon)$.



\begin{proposition}
\label{th:diffusion-scons-j}
Let $y\in\bb R^m$ and $j\in[n]$.
\begin{itemize}
\item If the set $\{\,\sgn a_i\mid i\in I(y,\max y)\,\}$ is consistent in coordinate~$j$,
then $I(p_j(y),\max y)=I(y,\max y)$.
\item Otherwise, $I(p_j(y),\max y)=I(y,\max y)\setminus \{\,i\mid a_{ij}\neq0\,\}$.
\end{itemize}
\end{proposition}

\begin{proof}
Denote $y'=p_j(y)$ and
\[
\alpha^-=\displaystyle\max_{i\mid a_{ij}<0}y_i, \qquad
\alpha^+=\displaystyle\max_{i\mid a_{ij}>0}y_i, \qquad
\alpha'=\displaystyle\max_{i\mid a_{ij}<0}y'_i = \max_{i\mid a_{ij}\neq0}y'_i = \max_{i\mid a_{ij}>0}y'_i .
\]
It follows from the definition of~$p_j$ that $y'_i=y_i$ for all~$i$
with $a_{ij}=0$ and
that $\alpha'$ lies between $\alpha^-$ and $\alpha^+$.

By Definition~\ref{def:SC}, the set
$\{\,\sgn a_i\mid i\in I(y,\max y)\,\}$ is consistent in coordinate~$j$ iff
\begin{equation}
\exists i\in I(y,\max y){:}\ a_{ij}<0 \quad\Longleftrightarrow\quad \exists i'\in I(y,\max y){:}\ a_{i'j}>0 .
\label{eq:argmax-SC'}
\end{equation}
Using the above observations, we see that:
\begin{itemize}
\item If both sides of~\refeq{argmax-SC'} are true, then
$\alpha^-=\alpha^+=\alpha'$, hence $y'=y$, hence $I(y',\max y)=I(y,\max y)$.
\item If both sides of~\refeq{argmax-SC'} are false, then
$\alpha^-,\alpha^+,\alpha'<\max y$. Hence 
$y_i,y_i'<\max y$ for all~$i$ with $a_{ij}\neq0$. Hence $I(y',\max y)=I(y,\max y)$.
\item If, say, the LHE of~\refeq{argmax-SC'} is true and
the RHS is false, then $\alpha^-=\max y$ and
$\alpha^+,\alpha'<\max y$. Hence $y'_i<\max
y$ for all~$i$ with $a_{ij}\neq0$.
Hence $I(y',\max y)=I(y,\max y)\setminus \{\,i\mid a_{ij}\neq0\,\}$.
\qedhere
\end{itemize}
\end{proof}

\begin{proposition}
\label{th:diffusion-scons}
For every $y\in\bb R^m$ and every $k\ge m$ we have
\[
\{\, \sgn a_i\mid i\in I(p^k(y),\max y)\,\} =
\cons\{\, \sgn a_i \mid i\in I(y,\max y) \,\} .
\]
\end{proposition}

\begin{proof}
Proposition~\ref{th:diffusion-scons-j} shows that the algorithm in
fact enforces consistency (see Definition~\ref{def:SCalg}) of the
set $\{\, \sgn a_i\mid i\in I(p^k(y),\max y)\,\}$. Moreover, in every
$m$~iterations each coordinate is visited and hence the set shrinks
unless it is already consistent. As the set initially has no
more than $m$~elements, after $m$~applications of~$p$ the set stops
shrinking and becomes consistent.
\end{proof}

\begin{proposition}
\label{th:diffusion-abc}
Let $y\in\bb R^m$.
\begin{itemize}
\item[(a)] If $\cons\{\, \sgn a_i \mid i\in \argmax y \,\}\neq\emptyset$, then for every $k\in\bb N$ we have $\max p^k(y)=\max y$.
\item[(b)] If $\cons\{\, \sgn a_i \mid i\in \argmax y \,\}\neq\emptyset$, then the set $\{\, \sgn a_i \mid i\in \argmax p^m(y) \,\}$ is consistent.
\item[(c)] If $\cons\{\, \sgn a_i \mid i\in \argmax y \,\}=\emptyset$, then $\max p^m(y)<\max y$.
\end{itemize}
\end{proposition}

\begin{proof}
By Proposition~\ref{th:monotone}, for every $k\in\bb N$ we have
$\max p^k(y)\le\max y$. By~\refeq{I} we have:
\begin{itemize}
\item If $I(p^k(y),\max y)\neq\emptyset$, then $\max p^k(y)=\max y$.
\item If $I(p^k(y),\max y)=\emptyset$, then $\max p^k(y)<\max y$.
\end{itemize}
Noting that $\argmax y=I(y,\max y)$, the claims now follow from
Proposition~\ref{th:diffusion-scons}.
\end{proof}

\subsection{Continuity and Boundedness}
\label{sec:proof-prelim-pe}

\begin{proposition}
\label{th:univar-cont}
Let $a_1\...a_m\in\bb R$ be such that $\{\,\sgn a_i\mid
i\in[m]\,\}=\{-1,1\}$. Then the function $\xi{:}\ \bb R^m\to\bb R$
given by
\begin{equation}
\xi(y) = \xi(y_1\...y_m) = \argmin_{x\in\bb R} \max_{i\in[m]}(a_ix+y_i)
\label{eq:univar}
\end{equation}
is continuous.
\end{proposition}

\begin{proof}
On a neighborhood of any point $y\in\bb R^m$, function~$\xi$ depends only
on the coordinates for which the affine functions $a_ix+y_i$
are active at~$y$. Moreover, we can move the minimum to the origin
without loss of generality. Therefore, to show that
function~\refeq{univar} is continuous on~$\bb R^m$, it suffices to
show that the function of the form~\refeq{univar} is continuous at the
point $y=0$.

For any $y\in\bb R^m$, $\xi(y)$ is the $x$-coordinate of the
intersection of the graphs of two affine functions $a_ix+y_i$ and
$a_jx+y_j$, one with negative and one with positive slope. Thus,
$\xi(y)=(y_i-y_j)/(a_j-a_i)$ for some $(i,j)$ such that $a_i<0$ and
$a_j>0$. Therefore,
\[
\|y\|_\infty\le\delta
\quad\Longrightarrow\quad
|\xi(y)|
\le \max_{i\mid a_i<0} \max_{j\mid a_j>0} \frac{|y_i-y_j|}{a_j-a_i}
\le \delta \max_{i\mid a_i<0} \max_{j\mid a_j>0} \frac{2}{a_j-a_i} .
\]
This shows that function~$\xi$ is continuous at $y=0$.
\end{proof}

\begin{proposition}
\label{th:pj-cont}
For every $j\in[n]$, the map~$p_j$ is continuous.
\end{proposition}

\begin{proof}
Map~$p_j$ is continuous because it is a composition of function~$\xi$
from Proposition~\ref{th:univar-cont} and the affine map
$y_i\mapsto a_{ij}x_j+y_i$ (as given by Step~3 of coordinate
minimization).
\end{proof}

\begin{proposition}
\label{th:bounded-yk}
Let $0\in\rint\conv\{a_1\...a_m\}$.
Let $y\in\bb R^n$.
Then the sequence $(p^k(y))_{k\in\bb N}$ is bounded. 
\end{proposition}

\begin{proof}
By Proposition~\ref{th:monotone}, for every~$k$ we have $p^k(y)\le\max y$.
The claim now follows from Proposition~\ref{th:bounded-y}.
\end{proof}

\begin{proposition}
\label{th:Ieps}
Let $y,y'\in\bb R^m$, $\epsilon\ge0$, $\delta\ge0$, and
$\|y-y'\|_\infty\le\delta$. Then $\argmax\nolimits^\epsilon y \subseteq
\argmax\nolimits^{\epsilon+2\delta} y'$.
\end{proposition}

\begin{proof}
Let $\|y-y'\|_\infty\le\delta$, i.e., 
$
-\delta \le y_i-y'_i \le \delta
$
for every~$i$.
This implies
$
-\delta \le \max y - \max y' \le \delta .
$
By these inequalities, for every~$i$ we have the implication
\begin{equation}
\max y - y_i \le \epsilon
\quad\Longrightarrow\quad
\max y' - y'_i \le \epsilon+2\delta .
\label{eq:e-cont-proof}
\end{equation}
By~\refeq{argmax-eps}, this means that
$\argmax\nolimits^\epsilon y \subseteq
\argmax\nolimits^{\epsilon+2\delta} y'$.
\end{proof}

\begin{proposition}
\label{th:e-cont}
The function~$e$ is continuous.
\end{proposition}

\begin{proof}
Let $\|y-y'\|_\infty\le\delta$. By~\refeq{e}, the set $\{\,\sgn a_i \mid
i\in \argmax\nolimits^{e(y)} y\,\}$ has a non-empty consistent subset. By
Proposition~\ref{th:Ieps}, the set $\{\,\sgn a_i \mid i\in
\argmax\nolimits^{e(y)+2\delta} y'\,\}$ has the same consistent
subset, therefore $e(y')\le e(y)+2\delta$. Similarly we prove that $e(y)\le
e(y')+2\delta$. Thus $|e(y)-e(y')|\le2\delta$.
\end{proof}

\begin{proposition}
\label{th:bounded-e}
Let $\cons\{\sgn a_i\mid i\in[m]\}\neq\emptyset$.
Then for every $y\in\bb R^n$ we have $e(y)\le\max y-\min y$.
\end{proposition}

\begin{proof}
Let $\epsilon=\max y-\min y$. Then clearly $\argmax^\epsilon y=[m]$, hence $\cons\{\,\sgn a_i\mid i\in \argmax\nolimits^\epsilon y\,\}\neq\emptyset$.
\end{proof}

%

\subsection{Convergence}

Using the preparations from~\S\ref{sec:proof-prelim-alg}
and~\S\ref{sec:proof-prelim-pe}, we now prove the main convergence
result.  We will do it by
reformulating~\cite[Theorem~1]{Schlesinger-2011}.

Let $q=p^m$ denote $m$~cycles of coordinate minimization. In this
section, $y_k$ will denote a vector from~$\bb R^m$, rather than the
$k$-th component of a vector $y\in\bb R^m$.  Sequences such as
$(y_k)_{k\in\bb N}$ will be denoted in short as $(y_k)$.

Recall that an {\em accumulation point\/} of a sequence is the limit
point of its convergent subsequence.

\begin{theorem}
\label{th:accum}
Let $q{:}\ \bb R^m\to\bb R^m$ be continuous. Let $y\in\bb R^m$. Let
the sequence $(\max q^k(y))$ be convergent. Then every accumulation
point~$y^*$ of the sequence $(q^k(y))$ satisfies $\max q(y^*)=\max y^*$.
\end{theorem}

\begin{proof}
For brevity, denote $y_k=q^k(y)$. Let $y^*$~be an accumulation point
of $(y_k)$, thus
\begin{equation}
\lim_{l\to\infty} y_{k(l)} = y^*
\label{eq:accum-proof:1}
\end{equation}
for some strictly increasing function $k{:}\ \bb N\to\bb N$. Applying
the continuous map~$q$ to equality~\refeq{accum-proof:1} yields
\begin{equation}
\lim_{l\to\infty} q(y_{k(l)})
= \lim_{l\to\infty} y_{k(l)+1}
= q(y^*) ,
\label{eq:accum-proof:2}
\end{equation}
where we used that $q(y_{k(l)})=y_{k(l)+1}$.  Now
\begin{equation}
\max y^* = \lim_{l\to\infty} \max y_{k(l)} = \lim_{k\to\infty} \max y_k = \lim_{l\to\infty} \max y_{k(l)+1} = \max q(y^*) .
\label{eq:accum-proof:3}
\end{equation}
The first and last equality follow from applying the continuous
function $\max{:}\ \bb R^m\to\bb R$ (defined by~\refeq{max}) to
equalities~\refeq{accum-proof:1} and~\refeq{accum-proof:2}. The second
and third equality hold because the sequence $(\max y_k)$ is
convergent and thus every its subsequence converges to the same
number.
\end{proof}


%
%

The following fact is well-known from analysis:

\begin{proposition}
\label{th:subseq-conv}
Let $(a_k)$ be a bounded sequence. If every convergent subsequence of
$(a_k)$ converges to a point~$a$, then the sequence $(a_k)$
converges to~$a$.
\end{proposition}

\begin{proof}
Suppose $(a_k)$ does not converge to~$a$.  Then for some $\epsilon>0$,
for every~$k_0$ there is $k>k_0$ such that $\|a_k-a\|>\epsilon$.  So
there is a subsequence $(b_k)$ such that $\|b_k-a\|>\epsilon$ for
all~$k$.  As $(b_k)$~is bounded, by Bolzano-Weierstrass it has a
convergent subsequence, $(c_k)$. But $(c_k)$ clearly cannot converge
to~$a$, a contradiction.
\end{proof}

\begin{theorem}
\label{th:converge-sc}
Let $q{:}\ \bb R^m\to\bb R^m$ and $e{:}\ \bb R^m\to\bb R$ be
continuous such that, for every $y\in\bb R^m$:
\begin{enumerate}
\item $\max q(y)\le \max y$,
\item $\max q(y)=\max y$ implies $e(y)=0$,
\item the sequences $(q^k(y))$, $(\max q^k(y))$ and $(e(q^k(y)))$ are bounded.
\end{enumerate}
Then for every $y\in\bb R^m$ we have $\displaystyle\lim_{k\to\infty}e(q^k(y))=0$.
\end{theorem}

\begin{proof}
Denote $y_k=q^k(y)$. The sequence $(\max y_k)$ is bounded and
non-increasing, therefore convergent. By Theorem~\ref{th:accum}, every
accumulation point~$y^*$ of $(y_k)$ satisfies $\max q(y^*)=\max y^*$. This
implies $e(y^*)=0$. 

By Proposition~\ref{th:subseq-conv}, now it suffices to show that
every convergent subsequence of the sequence $(e(y_k))$ converges
to~0.  So let $(z_k)$ be a subsequence of $(y_k)$ such that
$\lim_{k\to\infty}e(z_k)=e^*$.
\begin{itemize}
\item If $(z_k)$ is convergent, then $y^*=\lim_{k\to\infty}z_k$ is an
accumulation point of $(y_k)$, therefore $e(y^*)=0$. Applying the continuous
function~$e$ to this limit yields
$e(y^*)=\lim_{k\to\infty}e(z_k)=e^*=0$.
\item If $(z_k)$ is not convergent, by Bolzano-Weierstrass it has a
convergent subsequence, $(w_k)$. As $(w_k)$~is also a subsequence of
$(y_k)$, by the above reasoning we have
$\lim_{k\to\infty}e(w_k)=0$. But because $\lim_{k\to\infty}e(z_k)=e^*$,
every subsequence of $(e(z_k))$ converges to~$e^*$. As $(w_k)$~is a
subsequence of~$(z_k)$, this implies $e^*=0$. \qedhere
\end{itemize}
\end{proof}

%
%
%



Theorem~\ref{th:converge-sc} implies
Theorem~\ref{th:e-converges}. Indeed, map~$q$ is continuous because it
is a composition of maps~$p_j$ which are continuous by
Proposition~\ref{th:pj-cont}. Function~$e$ is continuous by
Proposition~\ref{th:e-cont}. Condition~1 holds by
Proposition~\ref{th:monotone} and Condition~2 by
Proposition~\ref{th:diffusion-abc}(c). The sequences $(q^k(y))$ and
$(\max q^k(y))$ are bounded by Proposition~\ref{th:bounded-yk}. The
sequence $(e(q^k(y)))$ is bounded by Proposition~\ref{th:bounded-e}. 

We remark that Theorem~\ref{th:converge-sc} has a wider applicability,
to prove convergence to local consistency for several other
message-passing algorithms. For that, the functions $\max$, $q$
and~$e$ need to be replaced by appropriate functions in these
algorithms and they must satisfy the assumptions of he theorem.

\subsection*{Acknowledgement}
This work has been supported by the Czech Science Foundation grant 16-05872S.

\bibliographystyle{abbrv}
\bibliography{/home/werner/publications/bib/werner}

\end{document}